\documentclass{article}[12pt]
\usepackage{amsfonts,amsmath,latexsym,hyperref,graphicx,amssymb,amsthm,url,lineno,float}
\usepackage[abbrev,msc-links]{amsrefs}  
\usepackage{rotating}
\usepackage{mathtools}
\usepackage{tikz}
\usetikzlibrary{calc}
\usetikzlibrary{positioning}
\usepackage{color}

\newtheorem{theorem}{Theorem}[section] 
\newtheorem{lemma}[theorem] {Lemma}

\newtheorem{corollary}[theorem] {Corollary} 
\newtheorem{conjecture}[theorem] {Conjecture} 
\newtheorem{claim}[theorem] {Claim}
\theoremstyle{definition}

\title{The maximum number of cycles in a graph with fixed number of edges}
\author{Andrii Arman, Sergei Tsaturian\\ University of Manitoba}
\date{8 February 2017}

\begin{document}
\maketitle

\begin{abstract}
The main topic considered is maximizing the number of cycles in a graph with given number of edges. In 2009, Kir\'aly conjectured that there is constant $c$ such that any graph with $m$ edges has at most $(1.4)^m$ cycles. In this paper, it is shown that for sufficiently large $m$, a graph with $m$ edges has at most $(1.443)^m$ cycles. For sufficiently large $m$, examples of a graph with $m$ edges and $(1.37)^m$ cycles are presented. For a graph with given number of vertices and edges an upper bound on the maximal number of cycles is given. Also, exponentially tight bounds are proved for the maximum number of cycles in a multigraph with given number of edges, as well as in a multigraph with given number of vertices and edges.
\end{abstract}

\section{Introduction}
Counting the number of cycles in a graph is a problem that was studied for different classes of the graphs: graphs with given cyclomatic number, planar graphs, 3-regular and 4-regular graphs, and many others. However, only a few general bounds for number of cycles that use basic graph parameters are known. In this paper the bounds on the number of cycles in a graph as a function of number of vertices and edges are presented. 

Let $C(G)$ denote the number of cycles in a graph $G$. In 1897, Ahrens \cite{Ahrens} proved that for a graph $G$ with $n$ vertices, $m$ edges and $k$ components, 
\begin{align}\label{eq:oldbound}
m-n+k \leq C(G) \leq 2^{m-n+k}-1.
\end{align}

The lower bound in (\ref{eq:oldbound}) is tight; for example, it is achieved by any disjoint union of cycles and trees. The tightness of the upper bound in (\ref{eq:oldbound}) was shown by Mateti and Deo \cite{Mateti} and the only graphs for which the upper bound is tight are $K_3$, $K_4$, $K_{3,3}$ and $K_4-e$. Aldred and Thomassen \cite{AlTho} improved the upper bound in  (\ref{eq:oldbound}) by showing that for a connected graph $G$,

\begin{equation}\label{e:at}
C(G)\leq \frac{15}{16}2^{m-n+1}.
\end{equation}

Entringer and Slater \cite{Entringer} considered $C(G)$ for the class of connected graphs with fixed cyclomatic number $r=m-n+1$. It follows from the results of \cite{Entringer} that there is a $3$-regular connected graph $G$ for which $C(G)>2^{r-1}$. Shi \cite{Shi} presented an example of an outer-planar $3$-regular Hamiltonian graph $G$ with $C(G)=2^{r-1}+r-1$.

Kir\'{a}ly \cite{Kirali} investigated $C(G)$ for several classes of graphs: the union and the sum of two trees, 3-regular and 4-regular graphs, and graphs with average degree 4.  Kir\'{a}ly also conjectured that there is a constant $c$, such that for any graph $G$ that has $m$ edges, $$C(G)\leq c (1.4)^{m}.$$

Aldred and Thomassen \cite{AlTho} studied $C(G)$ for the class of planar graphs. Arman, Gunderson and Tsaturian \cite{AGT} studied $C(G)$ for the class of triangle-free graphs on $n$ vertices. In this paper, $C(G)$ is investigated for two other classes of graphs: those with $n$ vertices and $m$ edges, and those with $m$ edges.

The following notation is used. For $k\in \mathbb{Z}^+$, denote $\{i\in \mathbb{Z}; 1\leq i\leq k\}$ by $[k]$, and for a set $S$, denote $\{T \subseteq S : |T|=k\}$ by $[S]^k$. Graphs and multigraphs are defined as in \cite{bollobas}. Denote the average degree of a graph $G$ by $d(G)$, the maximum degree by $\Delta(G)$, and the minimum degree by $\delta(G)$.

Theorem~\ref{thm:newbound} implies that if graph $G$ has $n$ vertices and $m$ edges, then

\begin{equation}\label{eq:newbound}
C(G)\leq \begin{cases} \frac34\Delta(G)(\frac{m}{n-1})^{n-1}, \text{ if } \frac{m}{n-1}\geq 3,\\ \frac34 \Delta (G) \cdot (\sqrt[3]{3})^m, \text{ if } \frac{m}{n-1}<3.\end{cases}
\end{equation}

The bound in (\ref{eq:newbound}) is better than in (\ref{e:at}) for graphs with sufficiently large number of edges and average degree at least $4.25$.

For $m \in \mathbb{Z}^+$ let $C(m)$ be the maximum number of cycles in a graph with $m$ edges. In Corollary~\ref{cor:upperbound} it is shown that 
$$C(m)<8.25(\sqrt[3]{3})^m,$$
which for $m>4056$ implies
$$C(m)<1.443^m.$$

Theorem~\ref{thm:newbound} and Corollary~\ref{cor:upperbound} are proved in Section \ref{sec:mainthm}.
 
In Section~\ref{sec:minmax} it is shown that extremal graphs for $C(m)$ have bounded degrees. Namely, it is shown (Theorem~\ref{thm:maxdegree}) that if $G$ is a graph with $m$ edges with $C(G)=C(m)$, then $\Delta(G) \leq 11$.

In Section \ref{sec:example}, for $m$ sufficiently large, a graph $G$ with $m$ edges is constructed, such that 
$$C(G)\geq (2+\sqrt{8})^{\frac{m}{5}-1}\geq  1.37^m.$$
Corollary~\ref{cor:upperbound} and the result of Section \ref{sec:example} imply that for $m$ large enough,
\begin{equation}\label{eq:newbounds}
1.37^m \leq C(m) \leq 1.443^m.
\end{equation}

In Section~\ref{sec:multi}, the problems of maximizing number of cycles in multigraphs with given number of edges or given number of vertices and edges are considered. It is shown (Theorem~\ref{multiedges}) that if $G$ is a multigraph that has the most cycles among all multigraphs with $m$ "multi-edges", then 
$$ \frac{9}{10} (\sqrt[3]{3})^m\leq C(G)\leq 8.25 (\sqrt[3]{3})^m.$$

\section{Maximal degree of graphs with $C(m)$ cycles}\label{sec:minmax}
Recall that, for $m \in \mathbb{Z}^+$, $C(m)$ is the maximum number of cycles in a graph with $m$ edges.
\begin{theorem}\label{thm:maxdegree}
If $G$ is a graph with $m$ edges such that $C(G)=C(m)$, then $\Delta(G) \leq 11$.
\end{theorem}

The proof of Theorem \ref{thm:maxdegree} relies on the following two lemmas.

\begin{lemma}\label{lemma:deletethree}
Let $k\geq 6$ be a positive integer. For $1\leq i<j \leq k$, let $w_{i,j}$ be a non-negative real number, and let $S=\sum_{1\leq i<j \leq k}w_{i,j}$. Then there exists a 6-element set $D \subseteq [k]$ such that 
$$\sum_{ \scriptsize \begin{matrix}  1\leq i<j \leq k \\ i \not\in D , j \not\in D \end{matrix} }w_{i,j} \geq \left(1-\frac{6(2k-7)}{k(k-1)}\right)S.$$
\end{lemma}
\begin{proof}
The proof relies on an averaging argument. For each $i \in [k]$ set $w_i=\sum_{j \in [k], j\neq i}w_{i,j}$. Note that $$\sum_{i\in [k]}w_i=2S.$$ Let $X$ be a collection of all 6-element subsets of $[k]$. For $D\in X$ let 
\begin{align*}
S(D)&=\sum_{ \scriptsize \begin{matrix} 1\leq i<j \leq k \\ i \not\in D , j \not\in D \end{matrix} }w_{i,j}\\
&=S-\sum_{i \in D}\left(\sum_{j \in [k], j \neq i}w_{i,j} \right)+\sum_{ i, j \in D, i<j }w_{i,j}\\
&=S-\sum_{i \in D}w_{i}+\sum_{i, j \in D, i< j}w_{i,j}.
\end{align*}
Let $\overline{S(D)}$ be the average of $S(D)$ over all $D\in X$, then 
\begin{align*}
\overline{S(D)}&=\frac{\sum_{ D \in X } \left(S-\sum_{i\in D}w_i+\sum_{i, j \in D, i< j}w_{i,j}\right)}{ \binom{k}{6}}\\
&=S-\frac{\sum_{i \in [k]} \sum_{D \in X, i \in D}w_i}{\binom{k}{6}}+\frac{\sum_{1\leq i<j \leq k} \sum_{D \in X : \, i,j \in D}w_{i,j}}{\binom{k}{6}}\\
&=S-\frac{\sum_{i \in [k]} \binom{k-1}{5} w_i}{\binom{k}{6}}+\frac{\sum_{1\leq i<j \leq k} \binom{k-2}{4} w_{i,j}}{\binom{k}{6}}\\
&=S-\frac{\binom{k-1}{5}\cdot 2S}{\binom{k}{6}}+\frac{\binom{k-2}{4}\cdot S}{\binom{k}{6}}\\
&=\left(1-\frac{6(2k-7)}{k(k-1)}\right)S.\\
\end{align*}
There exists $D \in X$, such that $S(D) \geq \overline{S(D)}$, i.e.
$$\sum_{ \scriptsize \begin{matrix}  1\leq i<j \leq k \\ i \not\in D , j \not\in D \end{matrix} }w_{i,j} \geq \left(1-\frac{6(2k-7)}{k(k-1)}\right)S.$$
\end{proof}

\begin{lemma}\label{lemma:tripartition}
Let $k\geq 2$ be a positive integer. For $1\leq i<j \leq k$, let $w_{i,j}$ be a non-negative real number, and let $S=\sum_{1\leq i<j \leq k}w_{i,j}$. Then there exists a partition $A_1\cup A_2\cup A_3\cup A_4=[k]$, 
such that  

$$\sum_{1 \leq l < m \leq 4} \; \sum_{ \scriptsize \begin{matrix} i \in A_l\\ j \in A_m \end{matrix}}w_{i,j}\geq \left(\frac{3k^2-4}{4k(k-1)}\right) S.$$
\end{lemma}

\begin{proof}
For all $l \in [4]$ let $a_l=\lfloor \frac{k+l-1}{4}\rfloor$ (note that $a_1+a_2+a_3+a_4=k$). Let $X$ be the collection of all  ordered quadruples $(A_1, A_2, A_3, A_4)$, such that $\pi=A_1\cup A_2\cup A_3 \cup A_4$ is a partition of $[k]$ and for all $l\in [4]$, $|A_l|=a_l$. Note that 

$$\displaystyle |X|=\frac{k!}{ a_1! a_2!a_3!a_4!}.$$ For $p=(A_1,A_2,A_3,A_4)\in X$ define 

\begin{align*}
S(p)&=\sum_{1\leq l < m \leq 4} \; \sum_{ \scriptsize \begin{matrix} i \in A_l\\ j \in A_m \end{matrix}}w_{i,j}\\
&=S-\sum_{l \in [4]}\sum_{ \scriptsize \begin{matrix} i < j\\ i,j \in A_l \end{matrix}}w_{i,j}.
\end{align*}

Let $\overline{S(p)}$ be the average of $S(p)$ over all possible choices of $p$.
\begin{align*}
\overline{S(p)}&=\frac{\sum_{p \in X} (S-\sum_{l \in [4]}\sum_{ i,j \in A_l, i<j }w_{i,j})}{|X|}\\
&=S-\frac{\sum_{l \in [4]} \sum_{p \in X}\sum_{ i,j \in A_l, i<j }w_{i,j}}{|X|}\\
&=S-\frac{\sum_{l \in [4]} \sum_{ 1\leq i<j \leq k}\sum_{p \in X:\, i,j \in A_l }w_{i,j}}{|X|}\\
\end{align*}

Note that for any choice of $l \in [4]$ and any choice of $i,j$, such that $1\leq i<j \leq k$  there are exactly 
$$\frac{(k-2)!  (a_l)(a_l-1)}{ a_1! a_2 ! a_3 ! a_4 !} $$
quadruples $p \in X$, such that $i,j \in A_l$. Then, 

\begin{align*}
\overline{S(p)}&=S-(\sum_{l \in [4]} \sum_{ 1\leq i<j \leq k}\frac{(k-2)!  (a_l)(a_l-1)}{ a_1! a_2 ! a_3 ! a_4 !}w_{i,j} )\slash |X|\\
&=S-(\sum_{l \in [4]} \frac{(k-2)!  (a_l)(a_l-1)}{ a_1! a_2 ! a_3 ! a_4 !}\cdot S ) \cdot \frac{1}{|X|}\\
&=S-(\sum_{l \in [4]} \frac{(k-2)!(a_l)(a_l-1)}{ a_1! a_2 ! a_3 ! a_4 !}) \cdot S \cdot \frac{ a_1! a_2 ! a_3 ! a_4 !}{k!}\\
&=S-(\sum_{l \in [4]} \frac{\lfloor\frac{k+l-1}{4}\rfloor (\lfloor\frac{k+l-1}{4}\rfloor-1)}{k(k-1)})\cdot S\\
&=S\left(1-\frac{1}{k(k-1)}
 \cdot \begin{cases} \frac{k(k-4)}{4} , &\text{if } k\equiv 0 \text{ mod } 4 \\
   \frac{(k-1)(k-3)}{4}, &\text{if } k\equiv \pm 1 \text{ mod } 4\\
 \frac{(k-2)^2}{4} , &\text{if } k\equiv 2 \text{ mod } 4 \end{cases} \right)\\
&\geq S\left(1-\frac{(k-2)^2}{4k(k-1)}\right).\\
\end{align*}
There exists a $p=(A_1,A_2,A_3,A_4) \in X$, such that $S(p) \geq \overline{S(p)}$, therefore the partition $A_1 \cup A_2 \cup A_3 \cup A_4$ satisfies the statement of Lemma \ref{lemma:tripartition}.
\end{proof}

\begin{proof}[Proof of Theorem~\ref{thm:maxdegree}]
Let $m$ be a positive integer and $G$ be a graph with $m$ edges. To prove Theorem~\ref{thm:maxdegree}, it is sufficient to prove that if $\Delta(G)\geq 12$, then there is a graph $H$ with $m$ edges and with $C(H)>C(G)$.

Let $\Delta(G)\geq 12$ and $u$ be a vertex of maximal degree in $G$. Let $N(u)=\{u_1, u_2, \dots, u_k\}$ be the neighbourhood of $u$ (note that $k\geq 12$). For $1\leq i< j \leq k$, define $w_{i,j}$ to be the number of paths from the vertex $u_i$ to the vertex $u_j$ in the graph $G-u$. Then the number of cycles in graph $G$ that pass through vertex $u$ is $S=\sum_{1\leq i<j \leq k}w_{i,j}$. By Lemma \ref{lemma:deletethree}, there is a 6-element set $D=\{i_1, i_2, \dots,  i_6 \}$, such that 
\begin{equation}\label{eq:eqshort}
\sum_{ \scriptsize \begin{matrix} 1\leq i<j \leq k \\ i \not\in D , j \not\in D \end{matrix} }w_{i,j} \geq \left(1-\frac{6(2k-7)}{k(k-1)}\right)S.
\end{equation}
Suppose, upon re-indexing, that $D=\{k-5, k-4, \dots, k-1, k\}$. Applying Lemma \ref{lemma:tripartition} to the collection of real numbers $w_{i,j}$ with $1\leq i < j \leq k-6$ gives a partition $A_1\cup A_2 \cup A_3 \cup A_4=[k-6]$ with 
\begin{equation}\label{eq:eqlong}
\sum_{1 \leq l<m \leq 4} \; \sum_{\scriptsize \begin{matrix} i \in A_l\\ j \in A_m \end{matrix}}w_{i,j}\geq \left(\frac{3(k-6)^2-4}{4(k-6)(k-7)}\right)\left(1-\frac{6(2k-7)}{k(k-1)}\right)S.\end{equation}
For $i\in[4]$, let  $U_{i}=\{u_{j} \; : \; j\in A_i\}.$
Construct a graph $H$ by deleting $u$ and all of the edges incident to $u$, adding four new vertices $v_1$, $v_2$ , $v_3$, $v_4$, then for all $1\leq i \leq 4$ adding edges from $v_i$ to each vertex of $U_i$, and for all $1\leq i<j\leq 4$ adding edges $v_iv_j$ (see Figure \ref{fig:GtoH}).
Then $|E(H)|=|E(G)|$. 

\begin{figure}[h]
\begin{center}\begin{tikzpicture}[scale=0.5, line cap=round,line join=round,x=1.0cm,y=1.0cm]
\clip(-6,-1) rectangle (19,7);

\begin{scope}
\draw[line width=0.1pt, color=black, draw opacity=1] (0,3)--(-4,2.5);
\draw[line width=0.1pt, color=black, draw opacity=1] (0,3)--(4,2);
\draw[line width=0.1pt, color=black, draw opacity=1] (0,3)--(4,3);
\draw[line width=0.1pt, color=black, draw opacity=1] (0,3)--(-1.5,0.25);
\draw[line width=0.1pt, color=black, draw opacity=1] (0,3)--(-2.5,0.75);
\draw[line width=0.1pt, color=black, draw opacity=1] (0,3)--(1.5,0.25);
\draw[line width=0.1pt, color=black, draw opacity=1] (0,3)--(2.5,0.75);
\foreach \x in {0,...,5} \draw[line width=0.1pt, color=black, draw opacity=1] (0,3)--(-2.5+\x,5); 

\draw [fill=black] (0.,3.) circle (2pt);
\draw [fill=black] (-4,2.5) circle (2pt);
\draw [fill=black] (4.,2.) circle (2pt);
\draw [fill=black] (4.,3.) circle (2pt);
\draw [fill=black] (-1.5,0.25) circle (2pt);
\draw [fill=black] (-2.5,0.75) circle (2pt);
\draw [fill=black] (1.5,0.25) circle (2pt);
\draw [fill=black] (2.5,0.75) circle (2pt);

\foreach \x in {0,...,5}
\filldraw[black] (-2.5+\x,5) circle (2pt); 

\node [below,black] at (0,2.8) {\large{u}};
\node [below,black] at (-5,7) {\large{G:}};
\node [below,black] at (2.2,0.2) {\large{$U_3$}};
\node [below,black] at (-2.2,0.2) {\large{$U_2$}};
\node [left,black] at (-4.1,2.5) {\large{$U_1$}};
\node [right,black] at (4.1,2.5) {\large{$U_4$}};
\node [above,black] at (0,5.5) {\large{$D$}};

\draw (4,2.5) ellipse (0.2cm and 0.8cm);
\draw (-4,2.5) ellipse (0.2cm and 0.8cm);
\draw (2,0.5) circle [x radius=1cm, y radius=0.2cm, rotate=30];
\draw (-2,0.5) circle [x radius=1cm, y radius=0.2cm, rotate=-30];
\draw (0,5) ellipse (3cm and 0.3cm);

\end{scope}

\begin{scope}[xshift=12cm]
\draw[line width=0.1pt, color=black, draw opacity=1] (-2,3)--(-4,2.5);
\draw[line width=0.1pt, color=black, draw opacity=1] (2,3)--(4,2);
\draw[line width=0.1pt, color=black, draw opacity=1] (2,3)--(4,3);
\draw[line width=0.1pt, color=black, draw opacity=1] (-1,2)--(-1.5,0.25);
\draw[line width=0.1pt, color=black, draw opacity=1] (-1,2)--(-2.5,0.75);
\draw[line width=0.1pt, color=black, draw opacity=1] (1,2)--(1.5,0.25);
\draw[line width=0.1pt, color=black, draw opacity=1] (1,2)--(2.5,0.75);
\draw[line width=0.1pt, color=black, draw opacity=1] (1,2)--(-1,2);
\draw[line width=0.1pt, color=black, draw opacity=1] (2,3)--(-2,3);
\draw[line width=0.1pt, color=black, draw opacity=1] (1,2)--(2,3);
\draw[line width=0.1pt, color=black, draw opacity=1] (-1,2)--(-2,3);

\draw[line width=0.1pt, color=black, draw opacity=1] (1,2)--(-2,3);
\draw[line width=0.1pt, color=black, draw opacity=1] (-1,2)--(2,3);

\draw [fill=black] (-2.,3.) circle (2pt);
\draw [fill=black] (2.,3.) circle (2pt);
\draw [fill=black] (-1.,2.) circle (2pt);
\draw [fill=black] (1.,2.) circle (2pt);
\draw [fill=black] (-4,2.5) circle (2pt);
\draw [fill=black] (4.,2.) circle (2pt);
\draw [fill=black] (4.,3.) circle (2pt);
\draw [fill=black] (-1.5,0.25) circle (2pt);
\draw [fill=black] (-2.5,0.75) circle (2pt);
\draw [fill=black] (1.5,0.25) circle (2pt);
\draw [fill=black] (2.5,0.75) circle (2pt);

\foreach \x in {0,...,5}
\filldraw[black] (-2.5+\x,5) circle (2pt); 

\node [above,black] at (-2,3) {\large{$v_1$}};
\node [above,black] at (2,3) {\large{$v_4$}};
\node [left,black] at (-1.1,2) {\large{$v_2$}};
\node [right,black] at (1.1,2) {\large{$v_3$}};
\node [below,black] at (2.2,0.2) {\large{$U_3$}};
\node [below,black] at (-5,7) {\large{H:}};
\node [below,black] at (-2.2,0.2) {\large{$U_2$}};
\node [left,black] at (-4.1,2.5) {\large{$U_1$}};
\node [right,black] at (4.1,2.5) {\large{$U_4$}};
\node [above,black] at (0,5.5) {\large{$D$}};

\draw (4,2.5) ellipse (0.2cm and 0.8cm);
\draw (-4,2.5) ellipse (0.2cm and 0.8cm);
\draw (2,0.5) circle [x radius=1cm, y radius=0.2cm, rotate=30];
\draw (-2,0.5) circle [x radius=1cm, y radius=0.2cm, rotate=-30];
\draw (0,5) ellipse (3cm and 0.3cm);

\end{scope}
\end{tikzpicture}
\caption{Constructing graph $H$ from $G$.}\label{fig:GtoH}
\end{center}
\end{figure}
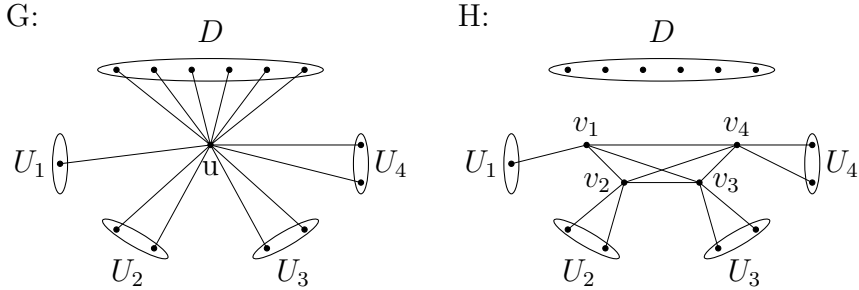

To count the number of cycles in $H$, note the following:
\begin{itemize}
\item Every cycle in $G$ that does not pass through the vertex $u$ is still a cycle in $H$. There are $C(G)-S$ such cycles. 
\item Let $C$ be a cycle in $G$ that for some $1\leq i<j \leq k-6$ contains a path $u_{i}uu_j$. If for some $l \in [4]$ $u_i$ and $u_j$ are in the same class $U_l$, then $C$ corresponds to the cycle in $H$ that uses the path $u_{i}v_lu_j$ instead of $u_iuu_j$. In the case if $u_i\in U_l$ and $u_j \in U_m$ for some $1\leq l<m \leq 4$, cycle $C$ corresponds to the cycle that uses the path $u_{i}v_lv_mu_j$ instead of $u_iuu_j$. By (\ref{eq:eqshort}), there are at least  $$\left(1-\frac{6(2k-7)}{k(k-1)}\right)S$$ cycles in $G$ that use path $u_iuu_j$ with $u_i, u_j \in N(u)\backslash D$.  
\item Every cycle in $G$ that for some $i\in A_l$ and $j \in A_m$ with $l\neq m$ contains a path $u_{i}uu_j$ gives rise to additional 4 cycles in $H$(except the one containing $u_{i}v_lv_mu_j$). For example, if $l=1$, $m=2$ the four new cycles contain paths $u_{i}v_1v_3v_2u_j$, $u_{i}v_1v_4v_2u_j$, $u_{i}v_1v_3v_4v_2u_j$ and 
$u_{i}v_1v_4v_3v_2u_j$ instead of $u_iuu_j$. According to (\ref{eq:eqlong}), there are at least $$\left(\frac{3(k-6)^2-4}{4(k-6)(k-7)}\right)\left(1-\frac{6(2k-7)}{k(k-1)}\right)S=\left(\frac{3k^2-36k+104}{4k(k-1)} \right)S$$
cycles in $G$ that for some $i\in A_l$ and $j \in A_m$ with $l\neq m$ pass through a path $u_{i}uu_j$.
\item There are 7 new cycles in $H$ spanned by the vertices $v_1, v_2, v_3, v_4$.
\end{itemize}
By all of the observations above, the number of cycles in $H$ is 
\begin{align*}
C(H)&\geq C(G)-S+\left(1-\frac{6(2k-7)}{k(k-1)}\right)S+4\left(\frac{3k^2-48k+104}{4k(k-1)} \right)S+7\\
&=C(G)+7+S\left(\frac{3(k-4)(k-12)}{k(k-1)}\right)\\
&>C(G).
\end{align*}
Therefore, $H$ has more cycles than $G$. 
\end{proof}

Note, that for $m=7$ the graphs that have the most cycles are $K_4$ plus an edge and $K_4$ with one edge replaced by a path of length two. In the first case minimum degree is one, in the second case -- two.\\
The authors can also prove the following theorem (that does not have direct relation to the main results of this paper).
\begin{theorem}\label{thm:mindegree1}
If $m>7$ and $G$ is a connected graph with $C(G)=C(m)$, then $\delta(G) \geq 3$.
\end{theorem}

\section{Cycles in graphs or multigraphs with fixed number of vertices and edges}\label{sec:mainthm}
Multigraphs are defined as in \cite{bollobas}. The degree $deg_G(V)$ of a vertex $v\in V(G)$ is the number of edges incident to $v$. For two vertices $u,v\in V(G)$, denote by $E(u,v)$ the set of all edges between $u$ and $v$. For a vertex $v\in V(G)$, denote by $N(v)$ the set of all vertices connected with $v$ by at least one edge. A cycle in a multigraph $G$ is a set of $k\geq 2$ distinct vertices and $k$ distinct edges $\{v_1,e_1,v_2,e_2,\dots ,e_k,v_1\}$, where for each $i \in [k]$, $v_i\in V(G), e_i\in E(G)$ and any consecutive vertex and edge are incident. As in the case of simple graphs, denote the number of cycles in a multigraph $G$ by $C(G)$. No loop can be a part of a cycle, hence only multigraphs without loops are considered.\\
The main result of this section is an upper bound for number of cycles in a graph (or multigraph) with fixed number of vertices and edges.
\begin{theorem}\label{thm:newbound}
Let $G$ be a multigraph with $n\geq 2$ vertices and $m$ edges.\\
If $\frac{m}{n-1}<3$, then
$$ C(G)<\frac34 \Delta (G) \cdot (\sqrt[3]{3})^m.$$
If $\frac{m}{n-1}\geq 3$, and $\lfloor\frac{m}{n-1}\rfloor=s$, $\alpha=\frac{m}{n-1}-s$, then
$$C(G)<\frac34\Delta(G)(s^{1-\alpha}(s+1)^{\alpha})^{n-1}=\frac34\Delta(G)((s^{1-\alpha}(s+1)^{\alpha})^{\frac{1}{s+\alpha}})^m.$$
\end{theorem}
To prove Theorem \ref{thm:newbound}, some notations and lemmas are needed.\\
Let $G$ be a multigraph with $n$ vertices. For vertices $v_1,\dots,v_k\in V(G)$, define  $F(v_1,v_2,\dots,v_k)=N(v_k)\backslash\{v_1,\dots,v_{k-1}\}$ and define\\ $f(v_1,\dots,v_k)=\max\{deg_{G-\{v_2,\dots,v_{k-1}\}}(v_k),1\}$. Denote the number of cycles in $G$ that contain the path $v_1e_1v_2\ldots e_{k-1}v_k$ by $C(v_1e_1v_2\ldots e_{k-1}v_k)$ (note that $C(v_1)$ is a number of cycles containing the vertex $v_1$). For brevity, write $F_k=F(v_1,\dots,v_k)$, $f_k=f(v_1,\dots,v_k)$, $C_k=C(v_1e_1\dots e_{k-1}v_k)$.
\begin{lemma}\label{l:completecycle}
Let $G$ be a multigraph with $n\geq 2$ vertices, $k\in [n-1]$, and $v_1 e_1 v_2 e_2 \dots v_{k-1}$ be a path in $G$. If $F_k\neq\emptyset$, then
\[C_k\leq f_k\cdot\max_{\substack{\\k+1\leq t\leq n\\v_{k+1}\in F_k\\\begin{turn}{90}\dots\end{turn}\\v_t\in F_{t-1}}} \{f_{k+1}\cdot f_{k+2} \cdot \ldots \cdot f_t\}.\]
\end{lemma}
\begin{proof}
Fix $n\geq 2$. Let $G$ be a multigraph on $n$ vertices. The proof is by mathematical induction on $l=n-k$.\\
\textit{Base case.}
Let $l=1$. Let $v_1 e_1 \dots v_{n-1}$ be a path in $G$; $C_{n-1}$ is to be bounded.\\
The condition $F_{n-1}\neq\emptyset$ means that $F_{n-1}=\{v_n\}$ and it remains to be proved that $C_{n-1}\leq f_{n-1}f_n$. Let $s$ be the number of edges between $v_{n-1}$ and $v_1$. Then $C_{n-1}\leq s+(f_{n-1}-s)f_n$. By definition, $f_n\geq 1$; therefore $s+(f_{n-1}-s)f_n\leq sf_n+(f_{n-1}-s)f_n=f_{n-1}f_n$, which proves the base case.\\
\textit{Inductive step.}
Let $i\in [n-1]$. Assume that the statement of the lemma holds for $l=i$, and prove it for $l=i+1$, i.e. let $v_1 e_1\dots v_{n-i-1}$ be a path in $G$, and $C_{n-i-1}=C(v_1e_1\dots e_{n-i-2}v_{n-i-1})$ is to be bounded.\\
Let $s$ be the number of edges between $v_{n-i-1}$ and $v_1$. Then
$$ C_{n-i-1}= s+\sum_{\mathclap{ \substack{v_{n-i}\in F_{n-i}\\ e_{n-i-1}\in E(v_{n-i},v_{n-i-1})  }}} \hspace{2mm} C (v_1e_1\dots v_{n-i-1}e_{n-i-1}v_{n-i}).$$

For all possible choices of $v_{n-i}$ and $e_{n-i-1}$, according to inductive hypothesis, 
\begin{align*}
C (v_1e_1\dots e_{n-i-1}v_{n-i})&\leq \begin{cases} f_{n-i}\max_{\substack{\\n-i+1\leq t\leq n\\v_{n-i+1}\in F_{n-i}\\\begin{turn}{90}\dots\end{turn}\\v_t\in F_{t-1}}} \{f_{n-i+1}\dots f_t\}, \text{ if }F_{n-i}\neq \emptyset \\ f_{n-i}, \text{ if } F_{n-i}=\emptyset \end{cases}\\
&\leq \max_{\substack{\\n-i\leq t\leq n\\v_{n-i}\in F_{n-i-1}\\\begin{turn}{90}\dots\end{turn}\\v_t\in F_{t-1}}} \{f_{n-i}\dots f_t\}.
\end{align*}
Therefore, 
\begin{align*}
C_{n-i-1}& \leq s+(f_{n-i-1}-s)\cdot\max_{\substack{\\n-i\leq t\leq n\\v_{n-i}\in F_{n-i-1}\\\begin{turn}{90}\dots\end{turn}\\v_t\in F_{t-1}}} \{f_{n-i}\dots f_t\}\\
&\leq f_{n-i-1}\cdot\max_{\substack{\\n-i\leq t\leq n\\v_{n-i}\in F_{n-i-1}\\\begin{turn}{90}\dots\end{turn}\\v_t\in F_{t-1}}} \{f_{n-i}\dots f_t\}.
\end{align*}
This proves that the statement of the lemma holds for $l=i+1$, and therefore by induction it holds for all $l\in [n-1]$.
\end{proof}

\begin{lemma}\label{l:maxdeg}
Let $G$ be a multigraph with $n\geq 3$ vertices and $m$ edges, and let $v_1$ be a vertex in $G$ of degree $\Delta (G)$.\\
If $\frac{m}{n-1}\geq 3$, and $\lfloor\frac{m}{n-1}\rfloor=s$, $\frac{m}{n-1}-s=\alpha$,  then there are at most\\ $\frac{\Delta(G)}{2}(s^{1-\alpha}(s+1)^{\alpha})^{n-1}$ cycles in $G$ that contain $v_1$.\\
If $\frac{m}{n-1}<3$, then there are at most $\frac{\Delta (G)}{2}\cdot(\sqrt[3]{3})^m$ cycles in $G$ that contain $v_1$.
\end{lemma}

\begin{proof}
Let $G$ be a multigraph with $n\geq 3$ vertices and $m$ edges, and $v_1$ be a vertex with degree $\Delta(G)$.\\
For any edge $e=v_1v_2$ incident to $v_1$, by Lemma \ref{l:completecycle}, the number of cycles that contain $e$ is at most 
$$f_2 \cdot \max_{\substack{\\3\leq t\leq n\\v_3\in F_2\\\begin{turn}{90}\dots\end{turn}\\v_t\in F_{t-1}}} \{f_3\dots f_t\} \leq \max_{\substack{\\2\leq t\leq n\\v_2\in F_1\\\begin{turn}{90}\dots\end{turn}\\v_t\in F_{t-1}}} \{f_2\dots f_t\}.$$
Every cycle through $v_1$ contains two such edges, therefore the number of cycles that contain $v_1$ is at most
\begin{equation}\label{eq:f}\frac{\Delta}{2}\cdot \max_{\substack{\\2\leq t\leq n\\v_2\in F_1\\\begin{turn}{90}\dots\end{turn}\\v_t\in F_{t-1}}} \{f_2\dots f_t\}\end{equation}
Let $v_2,\dots v_t$ be a collection of vertices that give the maximum in (\ref{eq:f}) with the smallest possible $t$. Then $f_t\geq 2$ (otherwise remove all $f_i=1$ after the last $f_k\geq 2$ to obtain the smaller collection of vertices that gives maximum in (\ref{eq:f})). Then for all $2 \leq i \leq t$,
$$f_i=deg_{G-\{v_2,\dots,v_{i-1}\}}(v_i).$$ 
For $2\leq i\leq t$, all the edge sets $\{v_iu\in E(G): u\in V(G)\backslash\{v_2,\dots,v_i\}\}$ are mutually disjoint, so $f_2+\dots+f_t\leq m$. Therefore, 
$$\frac{\Delta}{2} f_2 \cdot \hdots \cdot f_t \leq \frac{\Delta}{2}\cdot\max_{\substack{\\2\leq t\leq n\\x_2+\ldots+x_t\leq m,\\\forall i\in[2,t], x_i \in \mathbb{Z}^+} } \{ x_2 \cdot x_3 \cdot \hdots \cdot x_t\}.$$
So the number of cycles in $G$ that contain $v_1$ is at most
\begin{equation}\label{eq:x}\frac{\Delta}{2}\cdot\max_{\substack{\\2\leq t\leq n\\x_2+\ldots+x_t\leq m,\\\forall i\in[2,t], x_i \in \mathbb{Z}^+} } \{ x_2 \cdot x_3 \cdot \hdots \cdot x_t\}. \end{equation}
For a fixed $t$ the product $x_2\dots x_t$ in (\ref{eq:x}) obtains its maximum when $x_i$s ($i\geq 2$) are as equal as possible (for all $i,j$ $|x_i-x_j|\leq 1$), and their sum is equal to $m$. Let $\lfloor\frac{m}{n-1}\rfloor=s$, $\frac{m}{n-1}=s+\alpha$.\\
If $s\geq 3$ (which is equivalent to $\frac{m}{n-1}\geq 3$), let the maximum in (\ref{eq:x}) be achieved for some $t\leq n$ and let $x_2, \cdots, x_t$ be a collection of $x_i$s that gives the maximum in (\ref{eq:x}). If $t<n$, then $s\geq 3$ implies that either for some $i\in[t]$, $x_{i}\geq 5$, or for two different $i,j\in[t]$, $x_i=x_j=4$. In the first case replacing $x_i$ by $x_i-2$ and setting $x_{t+1}=2$ gives a collection of $x_i$s with a bigger product. In the second case setting $x_i=x_j=3$ and $x_{t+1}=2$ increases the product of $x_i$s. Hence, the maximum in (\ref{eq:x}) is achieved when $t=n$.
For all $2\leq i \leq n$, $x_i=s$ or $x_i=s+1$. Then the number of cycles in $G$ that pass through $v_1$ is at most   
\[\frac{\Delta}{2} x_2\dots x_n =  \frac{\Delta}{2} s^{(1-\alpha)(n-1)}(s+1)^{\alpha (n-1)}=\frac{\Delta}{2}(s^{1-\alpha}(s+1)^{\alpha})^{n-1}.\]

If $s<3$, let the maximum of (\ref{eq:x}) be achieved for some $2\leq t\leq n$ and let $x_2, \cdots, x_t$ be the collection of $x_i$s that gives the maximum in (\ref{eq:x}). Recall that for all $i,j$ $|x_i-x_j|\leq 1$. If for two different $i,j\in[t]$ $x_i=x_j>3$, then $m> 6+3(t-2)=3t$, and $s<3$ implies that $t<n$. Replacing $x_i$ by $x_i-1$, $x_j$ by $x_j-1$ and setting $x_{t+1}=2$ increases the product. Therefore, there is at most one $i$, such that $x_i=4$. 
If there is $i$ such that $x_i=1$, then replacing any $x_j$ ($j \neq i$) by $x_j+1$ and deleting $x_i$ increases the product. If for some $i,j,k\in [t]$ $x_i=x_j=x_k=2$, then replacing $x_i$ by $3$, $x_j$ by 3 and deleting $x_k$ increases the product. Therefore, $\{x_2,\dots, x_t\}\in \left\{ \{3,3,\dots,3,2,2\},\{3,3,\dots,3,4\},\{3,3,\dots,3,2\}, \{3,3, \dots, 3\} \right\}$. Then $x_2 \dots x_t$ is at most $3^{\frac{m}{3}}$, so the number of cycles that pass through $v_1$ is at most  $$\frac{\Delta}{2} x_2\dots x_t\leq \frac{\Delta}{2} 3^{\frac{m}{3}}.$$

\end{proof}

\begin{proof}[Proof of the Theorem \ref{thm:newbound}]
The proof is by mathematical induction on $n$.\\
\textit{Base case}. If $n=2$, there is only one multigraph on $n$ vertices with $m$ edges -- two vertices connected by $m$ edges. In this case $s=\frac{m}{n-1}=m$, and $G$ has $\max\{\binom{m}{2},0\}$ cycles, which is less than $\frac{3}{4}m(\sqrt[3]{3})^m$ (for the case $m<3$), and less than $\frac34 m\cdot m$ (for the case $m\geq 3$).\\
\textit{Inductive step}. Let $k\geq 3$ be an integer, and suppose that the statement of the theorem is proved for $n=k-1$. Let $G$ be a multigraph with $k$ vertices, $m$ edges and let $v_1$ be a vertex of maximal degree in $G$.\\
Suppose that $\frac{m}{k-1}<3$.\\
If $\Delta(G)\leq 2$, then every edge is contained in at most one cycle, and every cycle contains at least two edges, so the number of cycles in $G$ is at most $$\frac{m}{2}\leq \frac34 \Delta(G)\cdot (\sqrt[3]{3})^m.$$
If $\Delta(G)\geq 3$, then the multigraph $G-v_1$ has at most $m-3$ edges, $\Delta(G-v_1)\leq \Delta(G)$ and $\frac{|E(G-v_1)|}{|V(G-v_1)|-1}\leq \frac{m}{k-1}<3$, therefore, by inductive assumption, the number of cycles in $G-v_1$ is at most $\frac{3}{4}\Delta(G)\cdot (\sqrt[3]{3})^{m-3}$. By Lemma~\ref{l:maxdeg}, the number of cycles that contain $v_1$ is at most $\frac{\Delta (G)}{2}\cdot(\sqrt[3]{3})^m$, therefore the total number of cycles in $G$ is at most
$$\frac{\Delta (G)}{2}\cdot(\sqrt[3]{3})^m+\frac{3}{4}\Delta(G)\cdot (\sqrt[3]{3})^{m-3}=\frac34\Delta(G)\cdot (\sqrt[3]{3})^m.$$
Suppose that $\frac{m}{k-1}\geq3$.\\
Let $s=\lfloor\frac{m}{k-1}\rfloor$, $\alpha=\frac{m}{k-1}-\lfloor\frac{m}{k-1}\rfloor$. Note that $\Delta(G-v_1)\leq \Delta (G)$ and let $$y=\frac{|E(G-v_1)|}{|V(G-v_1)|-1}\leq\frac{m}{k-1}.$$
Note that the function $$f(x)=(\lfloor x\rfloor)^{1-x+\lfloor x\rfloor}(\lfloor x\rfloor+1)^{x-\lfloor x\rfloor}$$ is non-decreasing on every interval $[a,a+1], a\in \mathbb{Z}_{\geq 0}$ (and hence on $\mathbb{R}^+$), therefore
\begin{equation}\label{eq:ss}s^{1-\alpha}(s+1)^{\alpha}\geq f(3)=3.\end{equation}
If $y\geq 3$, then, by the induction hypothesis,
\begin{align*}
|E(G-v_1)|& \leq\frac34\Delta(G)((\lfloor y\rfloor)^{1-y+\lfloor y\rfloor}(\lfloor y\rfloor+1)^{y-\lfloor y\rfloor})^{k-2}\\
&\leq\frac34\Delta(G)(s^{1-\alpha}(s+1)^{\alpha})^{k-2}.
\end{align*}
If $y<3$, then $|E(G-v_1)|<3(k-2)$, and by the induction hypothesis
\begin{align*}
    |E(G-v_1)|&\leq\frac34\Delta(G)(\sqrt[3]{3})^{|E(G-v_1)|}<\frac34\Delta(G)(\sqrt[3]{3})^{3(k-2)}\\
    &=\frac34\Delta(G)\cdot 3^{k-2}\leq\frac34\Delta(G)(s^{1-\alpha}(s+1)^{\alpha})^{k-2}.
\end{align*}
Hence, for any $y$, $|E(G-v_1)|\leq\frac34\Delta(G)(s^{1-\alpha}(s+1)^{\alpha})^{k-2}$, which together with Lemma~\ref{l:maxdeg} and (\ref{eq:ss}) implies that 
\begin{align*}
   C(G) &=\frac{3\Delta(G)}{4}(s^{1-\alpha}(s+1)^{\alpha})^{k-2}+\frac{\Delta(G)}{2}(s^{1-\alpha}(s+1)^{\alpha})^{k-1}\\
    &\leq\frac{3\Delta(G)}{4}(s^{1-\alpha}(s+1)^\alpha)^{k-1},\\
\end{align*}
which proves the inductive step and hence the theorem.
\end{proof}

A consequence of Theorem \ref{thm:newbound} is
\begin{corollary}\label{cor:upperbound}
For any integer $m$
$$C(m)<8.25(\sqrt[3]{3})^m.$$
\end{corollary}
\begin{proof}
Let $G$ be a graph with $n$ vertices and $m$ edges, such that $C(G)=C(m)$. Suppose that $\frac{m}{n-1}\geq 3$. Let $f(s,\alpha)=(s^{1-\alpha}(s+1)^{\alpha})^{\frac{1}{s+\alpha}}$, then for any $s>0$, $f(s,\alpha)$ is monotone in $\alpha$ and $\displaystyle \max_{s\in \mathbb{Z_+},\alpha\in [0,1)}f(s,\alpha)=\max_{s\in \mathbb{Z_+}} {s^{\frac{1}{s}}}=\sqrt[3]{3}$. This, together with Theorem \ref{thm:newbound} and Theorem~\ref{thm:maxdegree}, implies that for $s=\lfloor\frac{m}{n-1}\rfloor$ and $\alpha=\frac{m}{n-1}-\lfloor\frac{m}{n-1}\rfloor$ $$C(m)=C(G)<\frac34 \Delta(G)((s^{1-\alpha}(s+1)^{\alpha})^{\frac{1}{s+\alpha}})^m\leq 8.25(\sqrt[3]{3})^m.$$
If $\frac{m}{n-1}<3$, then, by Theorem~\ref{thm:maxdegree} and Theorem \ref{thm:newbound} ,
$$C(m)=C(G)<\frac34 \Delta(G) (\sqrt[3]{3})^m\leq 8.25(\sqrt[3]{3})^m.$$
\end{proof}

\section{Example of a graph with $(1.37)^m$ cycles}\label{sec:example}
For $n\geq 1$ let $H_n$ be the graph on $2n+2$ vertices with $$V(H_n)=\{u_1, u_2, \dots, u_{n+1}, v_1, v_2, \dots v_{n+1}\} \text{\hspace{1cm}and}$$ 
$$E(H_n)=\{u_{i}v_{j} : i,j \in [n+1], |i-j|\leq 1\} \cup \{u_{i}u_{i+1} : i \in [n]\}\cup \{v_{i}v_{i+1} : i \in [n]\}.$$ 
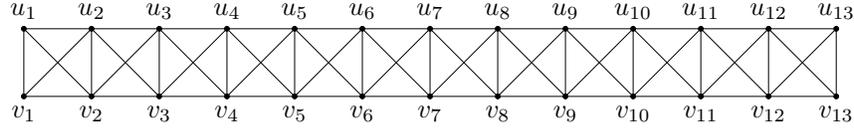
\begin{figure}[htbp]
\begin{center}\begin{tikzpicture}[scale=0.9, line cap=round,line join=round,x=1.0cm,y=1.0cm]
\clip(-6.3,-0.5) rectangle (6.3,1.5);
\foreach \a in {1,2,...,13} { %\a is the angle variable
\draw[fill] (-7+\a,0) circle (1pt); % 2cm is the radius; 1pt is the radius of the small bullet
}
\foreach \a in {1,2,...,13} { %\a is the angle variable
\draw[fill] (-7+\a,1) circle (1pt); % 2cm is the radius; 1pt is the radius of the small bullet
}

\foreach \a in {1,2,...,13} { %\a is the angle variable
\draw[line width=0.1pt, color=black, draw opacity=1] (-7+\a,0)--(-7+\a,1); % 2cm is the radius; 1pt is the radius of the small bullet
}

\foreach \a in {1,2,...,12} { %\a is the angle variable
\draw[line width=0.1pt, color=black, draw opacity=1] (-7+\a,0)--(-6+\a,0); % 2cm is the radius; 1pt is the radius of the small bullet
}

\foreach \a in {1,2,...,12} { %\a is the angle variable
\draw[line width=0.1pt, color=black, draw opacity=1] (-7+\a,1)--(-6+\a,1); % 2cm is the radius; 1pt is the radius of the small bullet
}

\foreach \a in {1,2,...,12} { %\a is the angle variable
\draw[line width=0.1pt, color=black, draw opacity=1] (-7+\a,0)--(-6+\a,1); % 2cm is the radius; 1pt is the radius of the small bullet
}

\foreach \a in {1,2,...,12} { %\a is the angle variable
\draw[line width=0.1pt, color=black, draw opacity=1] (-7+\a,1)--(-6+\a,0); % 2cm is the radius; 1pt is the radius of the small bullet
}

\foreach \a in {1,2,...,13} { %\a is the angle variable
\draw[fill] (-7+\a,0) circle (0pt) node[below] {$v_{\a}$} ;
}
\foreach \a in {1,2,...,13} { %\a is the angle variable
\draw[fill] (-7+\a,1) circle (0pt) node[above] {$u_{\a}$} ;
}

\end{tikzpicture}
\caption{Graph $H_{12}$.}\label{fig:H12}
\end{center}
\end{figure}

For $n\geq 1$ denote by $P(n)$ the number of paths from the vertex $u_1$ to the vertex $u_{n+1}$ in $H_n$. Note that $P(n)$ is also equal to the number of paths from $u_1$ to $v_{n+1}$ in $H_n$.

\begin{claim}
For all $n\geq 2$ 
$$P(n)=4P(n-1)+4P(n-2).$$
\end{claim}
\begin{proof}[Proof sketch]
The proof of the claim relies on an inductive argument and an observation that each path from $u_1$ to $u_{n+1}$ in $H_n$ corresponds to exactly one of the following paths:
\begin{itemize}
\item path from $u_1$ to $u_n$ in $H_{n-1}$ followed by the path $u_nu_{n+1}$ or by the path $u_nv_{n+1}u_{n+1}$. 
\item path from $u_1$ to $v_n$ in $H_{n-1}$ followed by the path $v_nu_{n+1}$ or by the path $v_nv_{n+1}u_{n+1}$. 
\item path from $u_1$ to $u_{n-1}$ in $H_{n-2}$ followed by the path $u_{n-1}u_{n}v_{n+1}v_nu_{n+1}$ or by the path $u_{n-1}v_{n}v_{n+1}u_nu_{n+1}$. 
\item path from $u_1$ to $v_{n-1}$ in $H_{n-2}$ followed by the path $v_{n-1}u_{n}v_{n+1}v_nu_{n+1}$ or by the path $v_{n-1}v_{n}v_{n+1}u_nu_{n+1}$. 
\end{itemize}
\end{proof}
Solving the recurrence relation leads to the inequality 
$$P(n)\geq (2+2\sqrt{2})^n.$$
Define the graph $G_n$ by identifying vertices $u_1$ and $u_n$ in $H_n$. Then $G_n$ has $2n+1$ vertices, $m=5n+1$ edges and  
$$C(G_n)\geq (2+2\sqrt{2})^{n}.$$

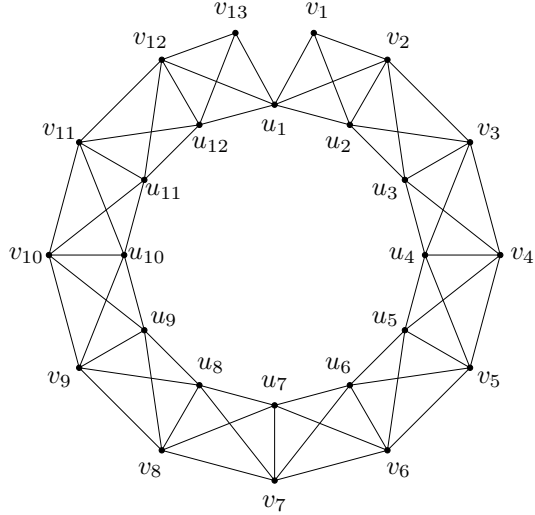
\begin{figure}[H]
\begin{center}\begin{tikzpicture}[scale=1, line cap=round,line join=round,x=1.0cm,y=1.0cm]
\clip(-5,-3.5) rectangle (4,3.5);
\foreach \a in {0,30,...,330} { %\a is the angle variable
\draw[fill] (\a+120:2cm) circle (1pt); % 2cm is the radius; 1pt is the radius of the small bullet
}
\foreach \a in {0,30,...,300} { %\a is the angle variable
\draw[fill] (\a+120:3cm) circle (1pt) ; % 2cm is the radius; 1pt is the radius of the small bullet
}

\foreach \a in {2,3,...,12} { %\a is the angle variable
\draw[fill] (-\a*30+120:3.3cm) circle (0pt) node {$v_{\a}$} ; % 2cm is the radius; 1pt is the radius of the small bullet
}

\foreach \a in {1,2,...,12} { %\a is the angle variable
\draw[fill] (-\a*30+120:1.7cm) circle (0pt) node {$u_{\a}$} ; % 2cm is the radius; 1pt is the radius of the small bullet
}

\foreach \a in {0,30,...,300} { %\a is the angle variable
\draw[line width=0.1pt, color=black, draw opacity=1] (\a+120:3cm)--(\a+120:2cm); % 2cm is the radius; 1pt is the radius of the small bullet
}

\foreach \a in {0,30,...,270} { %\a is the angle variable
\draw[line width=0.1pt, color=black, draw opacity=1] (\a+120:3cm)--(\a+30+120:2cm); % 2cm is the radius; 1pt is the radius of the small bullet
}

\foreach \a in {30,60,...,300} { %\a is the angle variable
\draw[line width=0.1pt, color=black, draw opacity=1] (\a+120:3cm)--(\a-30+120:2cm); % 2cm is the radius; 1pt is the radius of the small bullet
}

\foreach \a in {30,60,...,300} { %\a is the angle variable
\draw[line width=0.1pt, color=black, draw opacity=1] (\a+120:3cm)--(\a-30+120:3cm); % 2cm is the radius; 1pt is the radius of the small bullet
}

\foreach \a in {30,60,...,360} { %\a is the angle variable
\draw[line width=0.1pt, color=black, draw opacity=1] (\a+120:2cm)--(\a-30+120:2cm); % 2cm is the radius; 1pt is the radius of the small bullet
}
\draw[fill] (-20+120:3cm) circle (1pt);
\draw[fill] (-40+120:3cm) circle (1pt);
\draw[fill] (-20+120:3.3cm) circle (0pt) node {$v_{13}$} ;
\draw[fill] (-40+120:3.3cm) circle (0pt) node {$v_{1}$} ;
\draw[line width=0.1pt, color=black, draw opacity=1] (-30+120:2cm)--(0+120:3cm);
\draw[line width=0.1pt, color=black, draw opacity=1] (-30+120:2cm)--(-20+120:3cm);
\draw[line width=0.1pt, color=black, draw opacity=1] (-30+120:2cm)--(-40+120:3cm);
\draw[line width=0.1pt, color=black, draw opacity=1] (-30+120:2cm)--(-60+120:3cm);
\draw[line width=0.1pt, color=black, draw opacity=1] (-20+120:3cm)--(0+120:3cm);
\draw[line width=0.1pt, color=black, draw opacity=1] (-40+120:3cm)--(-60+120:3cm);
\draw[line width=0.1pt, color=black, draw opacity=1] (-0+120:2cm)--(-20+120:3cm);
\draw[line width=0.1pt, color=black, draw opacity=1] (-60+120:2cm)--(-40+120:3cm);

\end{tikzpicture}
\caption{$G_{12}$ with 25 vertices and 61 edges.}\label{fig:G12}
\end{center}
\end{figure}
For an integer $m$ let graph  $G$ be obtained from $G_{\lfloor \frac{m-1}{5} \rfloor}$ by adding $(m-5\lfloor \frac{m-1}{5} \rfloor-1)$ edges. Then $G$ has $m$ edges and for $m$ large enough
$$C(G) \geq C(G_{\lfloor\frac{m-1}{5}\rfloor}) \geq (2+2\sqrt{2})^{\lfloor\frac{m-1}{5}\rfloor} \geq (2+2\sqrt{2})^{\frac{m}{5}-1}> 1.37^m.$$

\section{Maximum number of cycles in multigraphs}\label{sec:multi}

The problems of maximizing the number of cycles with fixed number of edges or fixed average degree can be also considered for multigraphs. 
Using the techniques presented in this paper, the authors can prove the following two results.

\begin{theorem}\label{multidegree}
Let $G$ be a multigraph that has the maximum number of cycles among all the multigraphs with $n \geq 2$ vertices and $m\geq 3$ edges. Let $\lfloor\frac{m}{n-1}\rfloor=s$, $\alpha=\frac{m}{n-1}-s$.\\
If $\frac{m}{n-1}\geq 3$, then
$$\frac{8}{27}s(s^{1-\alpha}(s+1)^{\alpha})^{n-1}\leq C(G)\leq \frac34 \Delta(G) (s^{1-\alpha}(s+1)^{\alpha})^{n-1}.$$
If $\frac{m}{n-1}\leq 3$, then
$$ 4(\sqrt[3]{3})^{m-4} \leq C(G) < \frac34 \Delta (G) \cdot (\sqrt[3]{3})^m$$
\end{theorem}

The upper bounds in Theorem \ref{multidegree} follow from the Theorem \ref{thm:newbound}. For the lower bounds, define $C_{n,m}$ to be the multigraph obtained from the cycle $C_{n}$ by replacing each of some $m-\lfloor\frac{m}{n}\rfloor n$ consecutive edges with $\lfloor\frac{m}{n}\rfloor+1$ "multi-edges" and the rest $\lfloor\frac{m}{n}\rfloor n -m+n$ edges with $\lfloor\frac{m}{n}\rfloor$ "multi-edges". The lower bound in the first case is achieved by the graph $C_{n,m}$.  The lower bound in the second case is achieved by the graph $C_{\lfloor \frac{m+1}{3}\rfloor,m}$ with extra $n-\lfloor \frac{m+1}{3}\rfloor$ isolated vertices.

\begin{theorem}\label{multiedges}
Let $G$ be a multigraph with $m\geq 3$ edges that has the maximum number of cycles among all the multigraphs with $m$ edges. Then 
$$ \frac{9}{10} (\sqrt[3]{3})^m < 4(\sqrt[3]{3})^{m-4}\leq C(G)\leq 8.25 (\sqrt[3]{3})^m$$
\end{theorem}

The upper bound in Theorem \ref{multiedges} can be obtained by repeating the argument of Corollary \ref{cor:upperbound} and a version of Theorem \ref{thm:maxdegree}, modified for multigraphs. The example for the lower bound is the same as for the second case of Theorem~\ref{multidegree}.\\
Theorems \ref{multidegree} and \ref{multiedges} answer both questions for multigraphs up to a constant factor. The authors believe that for $m\geq 9$ the graph $C_{\lfloor \frac{m+1}{3} \rfloor,m}$ has the most cycles among all multigraphs with $m$ edges. 

\section{Concluding remarks}

Theorem \ref{thm:newbound} gives an upper bound for the number of cycles in a graph $G$ with $n$ vertices and $m$ edges. For a graph $G$ with $n$ vertices and average degree $d\geq 6$, Theorem~\ref{thm:newbound} implies
$$C(G) \leq 3\Delta(G)\left(\frac{d}{2}\right)^n.$$For $d=\Omega(\ln n)$, let $G$ be a random graph $G(n,p)$ with $p=\frac{d}{n-1}$. Glebov and Krivelevich~\cite{GK} proved that the number of cycles in $G$ is a.a.s. at least $ \left(\frac{d}{e}\right)^n(1+o(1))^n $. Therefore, if $G$ is a graph with the maximal number of cycles among all graphs with $n$ vertices and average degree $d$, then for $n$ large enough
$$\left(\frac{d}{e}\right)^n(1+o(1))^n\leq C(G)\leq (1+o(1))^n\left(\frac{d}{2}\right)^n.$$
This inequality and the fact that $C(K_n) \approx \frac{c}{\sqrt{n}}\left(\frac{n}{e}\right)^n$ for some constant $c$ (see \cite{AGT} for details) motivates the following conjecture.
\begin{conjecture}
For any $\alpha \in (0,1]$ and integer $n$ large enough any graph $G$ on $n$ vertices with average degree $d=\alpha n$ satisfies
$$C(G) \leq (1+o(1))^n\left( \frac{d}{e}\right)^n.$$
\end{conjecture}

As mentioned in the introduction, Theorem \ref{cor:upperbound} and the result of Section~\ref{sec:example} imply that $1.37^m\leq C(m)\leq 1.443^m$. 

Kir\'{a}ly \cite{Kirali} conjectured that $C(m)<1.4^m$. The upper bound in Corollary~\ref{cor:upperbound} is $8.25(\sqrt[3]{3})^m$,  which inspired the following conjecture.

\begin{conjecture}
For sufficiently large $m$, there exists a graph $G$ with $m$ edges and at least $(1+o(1))^m(\sqrt[3]{3})^m$ cycles.
\end{conjecture}

\section{Acknowledgments}
We would like to thank Karen Gunderson and Jamie Radcliffe for helpful discussions. We would also like to thank David Gunderson for valuable comments and suggestions.

\end{document}